\documentclass[leqno,11pt]{article}
\usepackage[utf8]{inputenc}
\usepackage[T1]{fontenc}
\usepackage{microtype}

\ifdefined\screenview
  \usepackage[ 
    paperheight=8.5in,     
    paperwidth=5.5in,
    margin=1.5em,
    includehead,
    includefoot
  ]{geometry}
\else    
  \usepackage[
    letterpaper
  ]{geometry}
\fi

\usepackage{amsmath}
\usepackage{amsthm}
\usepackage{amssymb}
\usepackage{tikz-cd}
\usetikzlibrary{arrows} 
\tikzset{
  commutative diagrams/.cd, 
  arrow style=tikz, 
  diagrams={>=stealth}
}
\usetikzlibrary{matrix,decorations.pathreplacing,calc}

\usepackage{textcomp}

\usepackage[sb]{libertine}
\usepackage[varqu,varl]{zi4}
\usepackage[libertine,bigdelims,vvarbb]{newtxmath} 
\usepackage[supstfm=libertinesups,%
  supscaled=1.2,%
  raised=-.13em]{superiors}
\useosf 
\usepackage[scr=boondox]{mathalfa}

\usepackage{slashed}
\usepackage{esint} 

\usepackage[english]{babel}

\usepackage{imakeidx}
\makeindex[intoc]

\usepackage{csquotes}
\usepackage[
  backend=biber,
  hyperref=true,
  backref=true,
  isbn=false,
  doi=true,
  natbib=true,
  eprint=true,
  useprefix=true,
  maxcitenames=99,
  maxbibnames=99,  
  maxalphanames=99, 
  minalphanames=99,
  safeinputenc,
  style=alphabetic,
  citestyle=alphabetic,
  block=space,
  datamodel=preamble/ext-eprint
]{biblatex}
\usepackage[
  hypertexnames = false,
  colorlinks    = true,
  citecolor     = gray,
  linkcolor     = gray,
  urlcolor      = gray,
  breaklinks
]{hyperref}
\makeatletter
\DeclareFieldFormat{arxiv}{%
  arXiv\addcolon\space
  \ifhyperref
    {\href{http://arxiv.org/\abx@arxivpath/#1}{%
       \nolinkurl{#1}%
       \iffieldundef{arxivclass}
         {}
         {\addspace\texttt{\mkbibbrackets{\thefield{arxivclass}}}}}}
    {\nolinkurl{#1}
     \iffieldundef{arxivclass}
       {}
       {\addspace\texttt{\mkbibbrackets{\thefield{arxivclass}}}}}}
\makeatother
\DeclareFieldFormat{mr}{%
  MR\addcolon\space
  \ifhyperref
    {\href{http://www.ams.org/mathscinet-getitem?mr=MR#1}{\nolinkurl{#1}}}
    {\nolinkurl{#1}}}
\DeclareFieldFormat{zbl}{%
  Zbl\addcolon\space
  \ifhyperref
    {\href{http://zbmath.org/?q=an:#1}{\nolinkurl{#1}}}
    {\nolinkurl{#1}}}
\renewbibmacro*{eprint}{%
  \printfield{arxiv}%
  \newunit\newblock
  \printfield{mr}%
  \newunit\newblock
  \printfield{zbl}%
  \newunit\newblock
  \iffieldundef{eprinttype}
    {\printfield{eprint}}
    {\printfield[eprint:\strfield{eprinttype}]{eprint}}
}
\AtEveryBibitem{%
  \clearlist{publisher}%
}
\DeclareFieldFormat[article,inproceedings,inbook,incollection]{title}{\textit{#1}}
\renewbibmacro{in:}{}
\newcommand{\printreferences}{\printbibliography[heading=bibintoc]}

\usepackage[inline,shortlabels]{enumitem}

\usepackage{subcaption}

\usepackage[yyyymmdd]{datetime}

\usepackage{etoolbox}
\ifundef{\abstract}{}{\patchcmd{\abstract}%
    {\quotation}{\quotation\noindent\ignorespaces}{}{}}

\usepackage[super]{nth}

\usepackage{aliascnt}

\numberwithin{equation}{section}

\renewcommand{\eqref}[1]{\hyperref[#1]{\rm(\ref*{#1})}}

\def\makeautorefname#1#2{\AtBeginDocument{\expandafter\def\csname#1autorefname\endcsname{#2}}}

\newcommand{\mynewtheorem}[2]{
  \newaliascnt{#1}{equation}          
  \newtheorem{#1}[#1]{#2}
  \aliascntresetthe{#1}
  \makeautorefname{#1}{#2}
}

\mynewtheorem{axiom}{Axiom}
\newtheorem*{axiom*}{Axiom}
\mynewtheorem{theorem}{Theorem}
\newtheorem*{theorem*}{Theorem}
\mynewtheorem{prop}{Proposition}
\newtheorem*{prop*}{Proposition}

\mynewtheorem{cor}{Corollary}
\mynewtheorem{construction}{Construction}
\mynewtheorem{lemma}{Lemma}
\mynewtheorem{conjecture}{Conjecture}
\newtheorem*{conjecture*}{Conjecture}
\mynewtheorem{hyp}{Hypothesis}
\newtheorem{step}{Step}
\newtheorem{substep}{Step}
\numberwithin{substep}{step}
\makeautorefname{step}{Step}
\makeautorefname{substep}{Step}

\numberwithin{subcase}{case}
\makeautorefname{case}{Case}
\makeautorefname{subcase}{case}

\usepackage{etoolbox}
\AtBeginEnvironment{proof}{\setcounter{step}{0}}

\theoremstyle{remark}
\mynewtheorem{remark}{Remark}
\newtheorem*{remark*}{Remark}
\mynewtheorem{convention}{Convention}
\newtheorem*{convention*}{Convention}
\newtheorem*{conventions*}{Conventions}

\theoremstyle{definition}
\mynewtheorem{definition}{Definition}
\newtheorem*{definition*}{Definition}
\mynewtheorem{notation}{Notation}
\mynewtheorem{data}{Data}
\mynewtheorem{example}{Example}
\newtheorem*{example*}{Example}
\mynewtheorem{exercise}{Exercise}
\mynewtheorem{solution}{Solution}
\mynewtheorem{question}{Question}
\mynewtheorem{problem}{Problem}
\newtheorem*{question*}{Question}

\mynewtheorem{summary}{Summary}

\makeautorefname{table}{Table}        
\makeautorefname{chapter}{Chapter}
\makeautorefname{section}{Section}
\makeautorefname{subsection}{Section}
\makeautorefname{subsubsection}{Section}
\makeautorefname{footnote}{Footnote}
\AtBeginDocument{\def\itemautorefname~#1\null{(#1)\null}}
\AtBeginDocument{\def\equationautorefname~#1\null{(#1)\null}}

\usepackage{bm}
\usepackage{mathtools} 
\usepackage{stmaryrd} 

\DeclareFontFamily{U}{mathx}{\hyphenchar\font45}
\DeclareFontShape{U}{mathx}{m}{n}{
      <5> <6> <7> <8> <9> <10>
      <10.95> <12> <14.4> <17.28> <20.74> <24.88>
      mathx10
      }{}
\DeclareSymbolFont{mathx}{U}{mathx}{m}{n}
\DeclareFontSubstitution{U}{mathx}{m}{n}
\DeclareMathAccent{\widecheck}{0}{mathx}{"71}
\DeclareMathAccent{\wideparen}{0}{mathx}{"75}

\DeclareMathOperator{\HF}{\HF}

\DeclareMathOperator{\Hol}{Hol}

\DeclareMathOperator{\im}{im}
\DeclareMathOperator{\ind}{index}

\DeclareMathOperator{\res}{res}

\DeclareMathOperator{\spec}{spec}

\def\({\left(}
\def\){\right)}
\def\<{\left\langle}
\def\>{\right\rangle}

\newcommand{\C}{{\mathbf{C}}}
\newcommand{\Gtwo}{G_2}

\newcommand{\N}{{\mathbf{N}}}

\newcommand{\PU}{{\P\U}}

\newcommand{\R}{\mathbf{R}}

\newcommand{\U}{\mathrm{U}}

\newcommand{\Z}{\mathbf{Z}}

\newcommand{\andq}{\text{and}\quad}

\newcommand{\co}{\mskip0.5mu\colon\thinspace}

\newcommand{\defined}[2][\key]{\def\key{#2}\textbf{#2}\index{#1}}
\newcommand{\delbar}{\bar{\del}}

\newcommand{\del}{\partial}

\newcommand{\floor}[1]{\lfloor#1\rfloor}

\newcommand{\into}{\hookrightarrow}
\newcommand{\iso}{\cong}

\newcommand{\loc}{\mathrm{loc}}

\newcommand{\qandq}{\quad\text{and}\quad}

\newcommand{\sEnd}{\mathrm{\sE nd}}

\renewcommand{\Im}{\operatorname{Im}}

\renewcommand{\P}{\mathbf{P}}
\renewcommand{\Re}{\operatorname{Re}}
\renewcommand{\det}{\operatorname{det}}

\renewcommand{\epsilon}{\varepsilon}
\renewcommand{\setminus}{{\backslash}}

\renewcommand{\leq}{\leqslant}
\renewcommand{\geq}{\geqslant}

\makeatletter
\renewcommand*\env@matrix[1][*\c@MaxMatrixCols c]{%
  \hskip -\arraycolsep
  \let\@ifnextchar\new@ifnextchar
  \array{#1}}

\renewcommand\xleftrightarrow[2][]{%
  \ext@arrow 9999{\longleftrightarrowfill@}{#1}{#2}}
\newcommand\longleftrightarrowfill@{%
  \arrowfill@\leftarrow\relbar\rightarrow}
\makeatother

\newcommand{\rd}{{\rm d}}

\newcommand{\ua}{{\underline a}}
\newcommand{\ub}{{\underline b}}

\newcommand{\cA}{\mathcal{A}}

\newcommand{\cH}{\mathcal{H}}

\newcommand{\cT}{\mathcal{T}}

\newcommand{\sE}{\mathscr{E}}

\newcommand{\sG}{\mathscr{G}}

\newcommand{\sM}{\mathscr{M}}

\newcommand{\fg}{{\mathfrak g}}

\newcommand{\fr}{{\mathfrak r}}

\newcommand{\fu}{{\mathfrak u}}

\author{
  Henrique Sá Earp
  \and
  Thomas Walpuski}

\title{$\Gtwo$--instantons over twisted connected sums}

\date{2015-11-03}

\begin{document}

\maketitle

\begin{abstract}
  We introduce a method to construct $\Gtwo$--instantons over compact $\Gtwo$--manifolds arising as the twisted connected sum of a matching pair of building blocks \cites{Kovalev2003,Kovalev2011,Corti2012a}.
  Our construction is based on gluing $\Gtwo$--instantons obtained from holomorphic vector bundles over the building blocks via the first named author's work \cite{SaEarp2011}.
  We require natural compatibility and transversality conditions which can be interpreted in terms of certain Lagrangian subspaces of a moduli space of stable bundles on a $K3$ surface.
\end{abstract}

\paragraph{Changes to the published version}
This article was first published in \href{http://dx.doi.org/10.2140/gt.2015.19.1263}{
Geometry and Topology, Volume 19, Issue 3, pp. 1263–1285 (2015)}.
The present version differs only in aesthetic aspects and \autoref{Remark_Examples} has been updated to reflect recent progress.

\section{Introduction}
\label{sec:intro}

A $\Gtwo$--manifold $(Y,g)$ is a Riemannian $7$--manifold whose holonomy group $\Hol(g)$ is contained in the exceptional Lie group $\Gtwo$ or, equivalently, a $7$--manifold $Y$ together with torsion-free $\Gtwo$--structure, that is, a non-degenerate $3$--form $\phi$ satisfying a certain non-linear partial differential equation, see, e.g., \cite[Part I]{Joyce1996}.
An important method to produce examples of compact  $\Gtwo$--manifolds with $\Hol(g)=\Gtwo$ is the \defined{twisted connected sum construction}, suggested by Donaldson, pioneered by \citet{Kovalev2003} and later extended and improved by \citet{Kovalev2011} and \citet{Corti2012a}.
Here is a brief summary of this construction:
A \defined{building block} consists of a projective $3$--fold $Z$ and a smooth anti-canonical $K3$ surface $\Sigma\subset Z$ with trivial normal bundle, see \autoref{def:building-block}.
Given a choice of hyperkähler structure $\(\omega_I,\omega_J,\omega_K\)$ on $\Sigma$ such that $\omega_J + i\omega_K$ is of type $(2,0)$ and $[\omega_I]$ is the restriction of a Kähler class on $Z$, one can make $V:=Z\setminus \Sigma$ into an asymptotically cylindrical (ACyl) Calabi--Yau $3$--fold, that is, a non-compact Calabi--Yau $3$--fold with a tubular end modelled on $\R_+\times S^1\times \Sigma$, see Haskins--Hein--Nordström~\cite{Haskins2012}.
Then $Y:=S^1\times V$ is an ACyl $\Gtwo$--manifold with a tubular end modelled on  $\R_+\times T^2\times \Sigma$.

\begin{definition}
  \label{def:matching-data}
  Given a pair of building blocks $(Z_\pm,\Sigma_\pm)$, a collection 
  \begin{equation*}
    \bm=\{\(\omega_{I,\pm},\omega_{J,\pm},\omega_{K,\pm}\),\fr\}
  \end{equation*}
  consisting of a choice of hyperkähler structures on $\Sigma_\pm$ such that $\omega_{J,\pm} + i\omega_{K,\pm}$ is of type $(2,0)$ and $[\omega_{I,\pm}]$ is the restriction of a Kähler class on $Z_\pm$ as well as a hyperkähler rotation $\fr\co\Sigma_+\to\Sigma_-$ is called \defined{matching data} and $(Z_\pm,\Sigma_\pm)$ are said to \defined{match} via $\bm$.
  Here a \defined{hyperkähler rotation} is a diffeomorphism $\fr\co \Sigma_+\to \Sigma_-$ such that
  \begin{equation}
    \label{eq:hyperkahler-rotation}
    \fr^*\omega_{I,-}=\omega_{J,+}, \quad
    \fr^*\omega_{J,-}=\omega_{I,+} \qandq
    \fr^*\omega_{K,-}=-\omega_{K,+}.
  \end{equation}
\end{definition} 

Given a matching pair of building blocks, one can glue $Y_\pm$ by interchanging the $S^1$--factors at infinity and identifying $\Sigma_\pm$ via $\fr$.
This yields a simply-connected compact $7$--manifold $Y$ together with a family of torsion-free $\Gtwo$--structures $(\phi_T)_{T \geq T_0}$, see \citet[Section 4]{Kovalev2003}.
From the Riemannian viewpoint $(Y,\phi_T)$ contains a ``long neck'' modelled on $[-T,T]\times T^2\times \Sigma_+$; one can think of the twisted connected sum as reversing the degeneration of the family of $\Gtwo$--manifolds that occurs as the neck becomes infinitely long.

If $(Z,\Sigma)$ is a building block and $\sE\to Z$ is a holomorphic vector bundle such that $\sE|_\Sigma$ is stable, then $\sE|_\Sigma$ carries a unique ASD instanton compatible with the holomorphic structure~\cite{Donaldson1985}.
The first named author showed that in this situation $\sE|_V$ can be given a Hermitian--Yang--Mills (HYM) connection asymptotic to the ASD instanton on $\sE|_\Sigma$ \cite{SaEarp2011}.
The pullback of a HYM connection over $V$ to $S^1\times V$ is a \defined{$\Gtwo$--instanton}, i.e., a connection $A$ on a $G$--bundle over a $\Gtwo$--manifold such that $F_A\wedge\psi=0$ with $\psi:=*\phi$.
It was pointed out by Simon Donaldson and Richard Thomas in their seminal article on gauge theory in higher dimensions \cite{Donaldson1998} that, formally, $\Gtwo$--instantons are rather similar to flat connections over $3$--manifolds; in particular, they are critical points of a Chern--Simons type functional and there is hope that counting them could lead to an enumerative invariant for $\Gtwo$--manifolds not unlike the Casson invariant for $3$--manifolds, see \cite[Section 6]{Donaldson2009} and \cite[Chapter 6]{Walpuski2013}.
The main result of this article is the following theorem, which gives conditions for a pair of such $\Gtwo$--instantons over $Y_\pm=S^1\times V_\pm$ to be glued to give a $\Gtwo$--instanton over $(Y,\phi_T)$.

\begin{theorem}
  \label{thm:itcs}
  Let $(Z_\pm,\Sigma_\pm)$ be a pair of building blocks that match via $\bm$.
  Denote by $Y$ the compact $7$--manifold and by $(\phi_T)_{T\geq T_0}$ the family of torsion-free $\Gtwo$--structures obtained from the twisted connected sum construction.
  Let $\sE_\pm\to Z_\pm$ be a pair of holomorphic vector bundles such that the following hold:
  \begin{itemize}
    \item $\sE_\pm|_{\Sigma_\pm}$ is stable.
      Denote the corresponding ASD instanton by $A_{\infty,\pm}$.

    \item There is a bundle isomorphism $\bar\fr\co\sE_+|_{\Sigma_+}\to \sE_-|_{\Sigma_-}$ covering the hyperkähler rotation $\fr$ such that $\bar\fr^*A_{\infty,-}=A_{\infty,+}$.
    
    \item There are no infinitesimal deformations of $\sE_\pm$ fixing the restriction to $\Sigma_\pm$:
    \begin{equation}
      \label{eq:no-deformations}
      H^1(Z_\pm,\sEnd_0(\sE_\pm)(-\Sigma_\pm))=0.
    \end{equation}

    \item 
      Denote by $\res_\pm\co H^1(Z_\pm,\sEnd_0(\sE_\pm)) \to H^1(\Sigma_\pm,\sEnd_0(\sE_\pm|_{\Sigma_\pm}))$ the restriction map and by
      \begin{equation*}
        \lambda_\pm\co
        H^1(Z_\pm,\sEnd_0(\sE_\pm))\to H^1_{A_{\infty,\pm}}
      \end{equation*}
      the composition of $\res_\pm$ with the isomorphism from \autoref{rmk:H1-ASD}.
      The images of $\lambda_+$ and $\bar\fr^*\circ\lambda_-$ intersect trivially in $H^1_{A_{\infty,+}}$:
    \begin{equation}
      \label{eq:trivial-intersection}
      \im\(\lambda_+\) \cap \im\(\bar\fr^*\circ\lambda_-\)=\{0\}.
    \end{equation}
  \end{itemize}
  Then there exists a non-trivial $\PU(n)$--bundle $E$ over $Y$, a constant $T_1\geq T_0$ and for each $T\geq T_1$ an irreducible and unobstructed\footnote{%
    See \autoref{def:irreducible-unobstructed}.
  }
  $\Gtwo$--instanton $A_T$ on $E$ over $(Y,\phi_T)$.
\end{theorem}

\begin{remark}
  \label{rmk:H1-ASD}
  If $A$ is an ASD instanton on a $\PU(n)$--bundle $E$ over a Kähler surface $\Sigma$ corresponding to a holomorphic vector bundle $\sE$, then
  \begin{equation*}
    H^1_A := \ker\(\rd_A^*\oplus\rd_A^+\co\Omega^1(\Sigma,\fg_E)\to(\Omega^0\oplus\Omega^+)(\Sigma,\fg_E)\)\iso H^1(\Sigma,\sEnd_0(\sE)),
  \end{equation*}
  see \citet[Section 6.4]{Donaldson1990}.
  Here $\fg_E$ denotes the adjoint bundle associated with $E$.
\end{remark}

\begin{remark}
  If 
  \begin{equation}
    \label{eq:H1=0}
    H^1(\Sigma_+,\sEnd_0(\sE_+|_{\Sigma_+}))=\{0\},
  \end{equation}
  then \eqref{eq:trivial-intersection} is vacuous.
  If, moreover, the topological bundles underlying $\sE_\pm$ are isomorphic, then the existence of $\bar\fr$ is guaranteed by a theorem of Mukai~\cite[Theorem~6.1.6]{Huybrechts1997}.
\end{remark}

Since $H^2(Z_\pm,\sEnd_0(\sE_\pm))\iso H^1(Z_\pm,\sEnd_0(\sE_\pm)(-\Sigma_\pm))$ vanish by \eqref{eq:no-deformations}, there is a short exact sequence 
\begin{multline*}
  0 \to H^1(Z_\pm,\sEnd_0(\sE_\pm))
    \xrightarrow{\res_\pm} H^1(\Sigma_\pm,\sEnd_0(\sE_\pm|_{\Sigma_\pm})) \\
    \to H^2(Z_\pm,\sEnd_0(\sE_\pm)(-\Sigma_\pm)) \to 0.
\end{multline*}
This sequence is self-dual under Serre duality.
It was pointed out by \citet[p.\,176 ff.]{Tyurin2008} that this implies that
\begin{equation*}
  \im \lambda_\pm \subset H^1_{A_{\infty,\pm}} 
\end{equation*}
is a complex Lagrangian subspace with respect to the complex symplectic structure induced by $\Omega_\pm:=\omega_{J,\pm}+i\omega_{K,\pm}$ or, equivalently, Mukai's complex symplectic structure on $H^1(Z_\pm,\sEnd_0(\sE_\pm))$.
Under the assumptions of \autoref{thm:itcs} the moduli space $\sM(\Sigma_+)$ of holomorphic vector bundles over $\Sigma_+$ is smooth near $[\sE_+|_{\Sigma_+}]$ and so are the moduli spaces $\sM(Z_\pm)$ of holomorphic vector bundles over $Z_\pm$ near $[\sE_\pm]$.
Locally, $\sM(Z_\pm)$ embeds as a complex Lagrangian submanifold into $\sM(\Sigma_\pm)$.
Since $\fr^*\omega_{K,-}=-\omega_{K,+}$, both $\sM(Z_+)$ and $\sM(Z_-)$ can be viewed as Lagrangian submanifolds of $\sM(\Sigma_+)$ with respect to the symplectic form induced by $\omega_{K,+}$.  Equation~\eqref{eq:trivial-intersection} asks for these Lagrangian submanifolds to intersect transversely at the point $[\sE_+|_{\Sigma_+}]$.
If one thinks of $\Gtwo$--manifolds arising via the twisted connected sum construction as analogues of $3$--manifolds with a fixed Heegaard splitting, then this is much like the geometric picture behind Atiyah--Floer conjecture in dimension three \cite{Atiyah1988}.

\begin{remark}
  The hypothesis \eqref{eq:trivial-intersection} appears natural in view of the above discussion.
  Assuming \eqref{eq:H1=0} instead would slightly simplify the proof, see \autoref{rmk:transversality};
  however, it would also substantially restrict the applicability of \autoref{thm:itcs} and, hence, the chance of finding new examples of $\Gtwo$--instantons because \eqref{eq:H1=0} is a very strong assumption. 
\end{remark}

\begin{remark}
  \label{Remark_Examples}
  Using \autoref{thm:itcs} in a situation with \eqref{eq:H1=0}, the first example of an irreducible and unobstructed $\Gtwo$--instanton over a twisted connected sum has been constructed by the second named author in \cite{Walpuski2015}.
  Recent joint work by Grégoire Menet, Johannes Nordström and the first named author \cite{Menet2015} constructs a further example of an irreducible and unobstructed $\Gtwo$--instanton using \autoref{thm:itcs} in a situation where \eqref{eq:H1=0} fails.
\end{remark}

\paragraph{Outline}
We recall the salient features of the twisted connected sum construction in \autoref{sec:tcs}.
The expert reader may wish to skim through it to familiarise with our notation.
The objective of \autoref{sec:gluing} is to prove \autoref{thm:gluing}, which describes hypotheses under which a pair of  $\Gtwo$--instantons over a matching pair of ACyl $\Gtwo$--manifolds can be glued.
Finally, in \autoref{sec:holomorphic} we explain how these hypotheses can be verified for $\Gtwo$--instantons obtained via the first named author's construction.
\autoref{thm:itcs} is then proved by combining \autoref{thm:gluing} and \autoref{thm:saearp} with \autoref{prop:diagram}.


\paragraph{Acknowledgements.}

We are grateful to Simon Donaldson for suggesting the problem solved in this article.
We thank Marcos Jardim.
Moreover, we thank the anonymous referee for many helpful comments and suggestions.
TW was supported by ERC Grant 247331 and Unicamp-Faepex grant 770/13.

\section{The twisted connected sum construction}
\label{sec:tcs}

In this section we review the twisted connected sum construction using the language introduced by Corti--Haskins--Nordström--Pacini~\cite{Corti2012a}.

\subsection{Gluing ACyl \texorpdfstring{$\Gtwo$}{G2}--manifolds}
\label{sec:gluing-acyl-g2-manifolds}

We begin with gluing matching pairs of ACyl $\Gtwo$--manifolds.
  
\begin{definition}
  Let $(Z,\omega,\Omega)$ be a compact Calabi--Yau $3$--fold.
  Here $\omega$ denotes the Kähler form and $\Omega$ denotes the holomorphic volume form.
  A $\Gtwo$--manifold $(Y,\phi)$ is called \defined{asymptotically cylindrical (ACyl)} with asymptotic cross-section $(Z,\omega,\Omega)$ if there exist a constant $\delta<0$, a compact subset $K\subset Y$, a diffeomorphism $\pi\co Y\setminus K\to\R_+\times Z$ and a $2$--form $\rho$ on $\R_+\times Z$ such that
  \begin{equation*}
    \pi_*\phi=\rd t\wedge\omega+\Re\Omega+\rd\rho
  \end{equation*}
  and
  \begin{equation*}
    \nabla^k\rho=O(e^{\delta t})
  \end{equation*}
  for all $k\in\N_0$.
  Here $t$ denotes the coordinate on $\R_+$.
\end{definition}

\begin{remark}
  Unfortunately, $Z$ is the customary notation both for building blocks and asymptotic cross-sections of ACyl $\Gtwo$--manifolds.
  To avoid confusion we point out that, unlike asymptotic cross-sections, building blocks always come in pair with a divisor, e.g., $(Z,\Sigma)$.
\end{remark}

\begin{definition}
  A pair of ACyl $\Gtwo$--manifolds $(Y_\pm,\phi_\pm)$ with asymptotic cross-sections $(Z_\pm,\omega_\pm,\Omega_\pm)$ is said to \defined{match} if there exists a diffeomorphism $f\co Z_+\to Z_-$ such that
  \begin{equation*}
    f^*\omega_-=-\omega_+ \quad\text{and}\quad
    f^*\Re\Omega_-=\Re\Omega_+.
  \end{equation*}
\end{definition}

Let $(Y_\pm,\phi_\pm)$ be a matching pair of ACyl $\Gtwo$--manifolds.
Fix $T\geq 1$.
Define $F\co [T,T+1]\times Z_+\to[T,T+1]\times Z_-$ by
\begin{equation*}
  F(t,z):=\(2T+1-t,f(z)\).
\end{equation*}

Denote by $Y_T$ the compact $7$--manifold obtained by gluing together 
\begin{equation*}
  Y_{T,\pm}:=K_\pm\cup \pi_\pm^{-1}\((0,T+1]\times Z_\pm\)
\end{equation*}
via $F$.
Fix a non-decreasing smooth function $\chi\co\R\to[0,1]$ \label{f:chi} with $\chi(t)=0$ for $t\leq 0$ and $\chi(t)=1$ for $t\geq 1$.
Define a $3$--form $\tilde\phi_T$ on $Y_T$ by
\begin{equation*}
  \tilde\phi_T:=\phi_\pm-\rd[\pi_\pm^*(\chi(t-T+1)\rho_\pm)]
\end{equation*}
on $Y_{T,\pm}$.
If $T\gg 1$, then $\tilde\phi_T$ defines a closed $\Gtwo$--structure on $Y_T$.
Clearly, all the $Y_T$ for different values of $T$ are diffeomorphic; hence, we often drop the $T$ from the notation.
The $\Gtwo$--structure $\tilde\phi_T$ is not torsion-free yet, but can be made so by a small perturbation:

\begin{theorem}[{\citet[Theorem~5.34]{Kovalev2003}}]
  \label{thm:kovalev}
  In the above situation there exist a constant $T_0\geq 1$ and for each $T\geq T_0$ there exists a $2$--form $\eta_T$ on $Y_T$ such that $\phi_T:=\tilde\phi_T+\rd\eta_T$ defines a torsion-free $\Gtwo$--structure;
  moreover, for some $\delta<0$
  \begin{equation}
    \label{eq:eta-estimate}
    \|\rd \eta_T\|_{C^{0,\alpha}}=O(e^{\delta T}).
  \end{equation}
\end{theorem}

\subsection{ACyl Calabi--Yau \texorpdfstring{$3$}{3}--folds from building blocks}

The twisted connected sum is based on gluing ACyl $\Gtwo$--manifolds arising as the product of ACyl Calabi--Yau $3$--folds with $S^1$.

\begin{definition}
  \label{def:ACylCY3}
  Let $(\Sigma,\omega_I,\omega_J,\omega_K)$ be a hyperkähler surface.
  A Calabi--Yau $3$--fold $(V,\omega,\Omega)$ is called \defined{asymptotically cylindrical (ACyl)} with asymptotic cross-section $(\Sigma,\omega_I,\omega_J,\omega_K)$ if there exist a constant $\delta<0$, a compact subset $K\subset V$, a diffeomorphism $\pi\co V\setminus K\to \R_+\times S^1\times \Sigma$, a $1$--form $\rho$ and a $2$--form $\sigma$ on $\R_+\times S^1 \times \Sigma$ such that
  \begin{align*}
    \pi_*\omega&=\rd t\wedge\rd \alpha + \omega_I + \rd\rho, \\
    \pi_*\Omega&=(\rd \alpha-i\rd t)\wedge(\omega_J+i\omega_K) + \rd\sigma
  \end{align*}
  and
  \begin{equation*}
    \nabla^k\rho=O(e^{\delta t}) \quad\text{as well as}\quad
    \nabla^k\sigma=O(e^{\delta t})
  \end{equation*}
  for all $k\in\N_0$.
  Here $t$ and $\alpha$ denote the respective coordinates on $\R_+$ and $S^1$.
\end{definition}

Given an ACyl Calabi--Yau $3$--fold $(V,\omega,\Omega)$, taking the product with $S^1$, with coordinate $\beta$, yields an ACyl $\Gtwo$--manifold 
\begin{equation*}
  (Y:=S^1\times V, \phi:=\rd\beta\wedge\omega+\Re\Omega)
\end{equation*}
with asymptotic cross-section
\begin{equation*}
  (T^2\times \Sigma,\rd\alpha\wedge\rd\beta+\omega_K,(\rd\alpha-i\rd\beta)\wedge(\omega_J+i\omega_I)).
\end{equation*}

Let $V_\pm$ be a pair of ACyl Calabi--Yau $3$--folds with asymptotic cross-section $\Sigma_\pm$ and suppose that $\fr\co \Sigma_+\to \Sigma_-$ is a hyperkähler rotation, see \eqref{eq:hyperkahler-rotation}.
Then $Y_\pm:=V_\pm\times S^1$ match via the diffeomorphism $f\co T^2\times \Sigma_+\to T^2\times \Sigma_-$ defined by
\begin{equation*}
  f(\alpha,\beta,x):=(\beta,\alpha,\fr(x)).
\end{equation*}

\begin{remark}
  If $f$ did not interchange the $S^1$--factors, then $Y$ would have infinite fundamental group and, hence, could not carry a metric with holonomy equal to $\Gtwo$ \cite[Proposition~10.2.2]{Joyce2000}.
\end{remark}

ACyl Calabi--Yau $3$--folds can be obtained from the following building blocks:

\begin{definition}[{\citet[Definition~5.1]{Corti2012}}]
  \label{def:building-block}
  A \defined{building block} is a smooth projective $3$--fold $Z$ together with a projective morphism $f\co Z\to \P^1$ such that the following hold:
  \begin{itemize}
  \item The anticanonical class $-K_Z\in H^2(Z)$ is primitive.
  \item $\Sigma:=f^{-1}(\infty)$ is a smooth $K3$ surface and $\Sigma \sim
    -K_Z$.
  \item If $N$ denotes the image of $H^2(Z)$ in $H^2(\Sigma)$, then the embedding $N\into H^2(\Sigma)$ is primitive.
  \item $H^3(Z)$ is torsion-free.
  \end{itemize}
\end{definition}

\begin{remark}
  The existence of the fibration $f\co Z\to \P^1$ is equivalent to $\Sigma$ having trivial normal bundle.
  This is crucial because it means that $Z\setminus \Sigma$ has a cylindrical end.
  The last two conditions in the definition of a building block are not essential; they have been made to facilitate the computation of certain topological invariants in \cite{Corti2012}.
\end{remark}

In his original work Kovalev~\cite{Kovalev2003} used building blocks arising from Fano $3$--folds by blowing-up the base-locus of a generic anti-canonical pencil.
This method was extended to the much larger class of semi Fano $3$--folds (a class of weak Fano $3$--folds) by \citet{Corti2012a}.
\citet{Kovalev2011} construct building blocks starting from $K3$ surfaces with non-symplectic involutions, by taking the product with $\P^1$, dividing by $\Z_2$ and blowing up the resulting singularities.

\begin{theorem}[{\citet[Theorem D]{Haskins2012}}]
  \label{thm:hhn}
  Let $(Z,\Sigma)$ be a building block and let $(\omega_I,\omega_J,\omega_K)$ be a hyperkähler structure on $\Sigma$ such that $\omega_J + i\omega_K$ is of type $(2,0)$.
  If $[\omega_I]\in H^{1,1}(\Sigma)$ is the restriction of a Kähler class on $Z$, then there is an asymptotically cylindrical Calabi--Yau structure $(\omega,\Omega)$ on $V:=Z\setminus \Sigma$ with asymptotic cross-section $(\Sigma,\omega_I,\omega_J,\omega_K)$.
\end{theorem}

\begin{remark}
  This result was first claimed by \citet[Theorem~2.4]{Kovalev2003}; see the discussion in~\cite[Section 4.1]{Haskins2012}.
\end{remark}

Combining the results of Kovalev and Haskins--Hein--Nordström, each matching pair of building blocks (see \autoref{def:matching-data}) yields a one-parameter family of $\Gtwo$--manifolds.
This is called the \defined{twisted connected sum construction}.

\begin{figure}
  \centering
  \begin{tikzpicture}
    \coordinate (l) at (0,0);
    \coordinate (r) at (4,0);
    \draw (l) to [out=180,in=0] +(-2,0) to [out=180,in=0] +(-1.2,.4) to [out=180, in=90] +(-.8,-.8) to [out=-90,in=180] +(.8,.-.8) to [out=0,in=180] +(+1.2,+.4) to [out=0,in=180] +(2,0);
    \draw [densely dotted,color=gray] (l) to +(.4,0);
    \draw [densely dotted,color=gray] (l)+(0,-.8) to +(.4,-.8);
    \draw [densely dotted,color=gray] (l)+(.4,0) arc (90:-90:0.1 and 0.4);
    \draw [dotted,color=gray] (l)+(.4,0) arc (90:270:0.1 and 0.4);
    \draw (l) arc (90:-90:0.1 and 0.4);
    \draw (l) [dashed,color=gray] arc (90:270:0.1 and 0.4);
    \draw (l)+(-.3,0) arc (90:-90:0.1 and 0.4);
    \draw (l)+(-.3,0) [dashed,color=gray] arc (90:270:0.1 and 0.4);
    \draw (l)+(-2.2,-1.2) node [right] {$\times$};
    \draw (l)+(-2.2,-1.6) node [right] {$S^1$};
    \draw (l)+(-3.1,-.4) node {$V_+$};
    \draw [decorate,decoration={brace,amplitude=3pt}] ($(l)+(-4,.6)$) to node [above=1mm] {$Y_{T,+}$} ($(l)+(0,.6)$);
    \draw (l)+(-.15,0) node [above] {\tiny $[T\!,\!T\!\!+\!\!1]$};

    \draw (r) to [out=0,in=180] +(2,0) to [out=0,in=180] +(1.2,.4) to [out=0,in=90] +(.8,-.8) to [out=-90,in=0] +(-.8,.-.8) to [out=180,in=0] +(-1.2,+.4) to [out=180,in=0] +(-2,0);
    \draw [densely dotted,color=gray] (r) to +(-.4,0);
    \draw [densely dotted,color=gray] (r)+(0,-.8) to +(-.4,-.8);
    \draw [densely dotted,color=gray] (r)+(-.4,0) arc (90:-270:0.1 and 0.4);
    \draw (r) arc (90:-270:0.1 and 0.4);
    \draw (r)+(.3,0) arc (90:-90:0.1 and 0.4);
    \draw (r)+(.3,0) [dashed,color=gray] arc (90:270:0.1 and 0.4);
    \draw (r)+(1.7,-1.2) node [right] {$\times$};
    \draw (r)+(1.7,-1.6) node [right] {$S^1$};
    \draw (r)+(+3.1,-.4) node {$V_-$};
    \draw [decorate,decoration={brace,amplitude=3pt}] ($(r)+(0,.6)$) to node [above=1mm] {$Y_{T,-}$} ($(r)+(4,.6)$);
    
    \draw (l)+(.6,0) node [right] {$\Sigma_+$};
    \draw (l)+(.6,-.4) node [right] {$\times$};
    \draw (l)+(.6,-.8) node [right] {$S^1$};
    \draw (r)+(-1.3,0) node [right] {$\Sigma_-$};
    \draw (r)+(-1.3,-.4) node [right] {$\times$};
    \draw (r)+(-1.3,-.8) node [right] {$S^1$};
    
    \draw [-stealth] ($(l)+(1.3,0)$) to node [above] {\scriptsize $\fr$} ($(r)+(-1.3,0)$);
    \draw [-stealth] ($(l)+(-1.5,-1.6)$) to ($(l)+(1.3,-1.6)$) to [out=0,in=180] ($(r)+(-1.3,-.8)$);
    \draw [color=white,line width=2pt] ($(l)+(1.3,-.8)$) to [out=0,in=180] ($(r)+(-1.3,-1.6)$) to ($(r)+(1.7,-1.6)$) ;
    \draw [-stealth] ($(l)+(1.3,-.8)$) to [out=0,in=180] ($(r)+(-1.3,-1.6)$) to ($(r)+(1.7,-1.6)$);
  \end{tikzpicture}
  \caption{The twisted connected sum of a matching pair of building blocks.}
\end{figure}
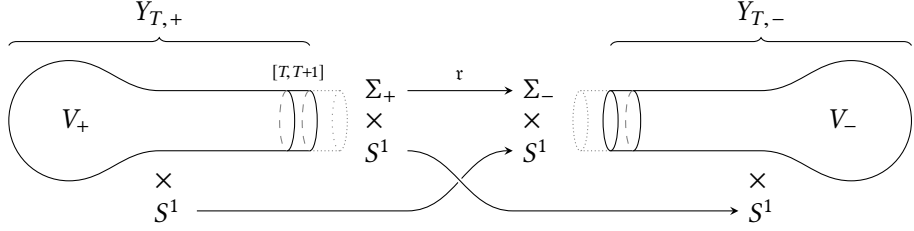



\section{Gluing \texorpdfstring{$\Gtwo$}{G2}--instantons over ACyl \texorpdfstring{$\Gtwo$}{G2}--manifolds}
\label{sec:gluing}

In this section we discuss when a pair of $\Gtwo$--instantons over a matching pair of ACyl $\Gtwo$--manifolds $Y_\pm$ can be glued to give a $\Gtwo$--instanton over $(Y,\phi_T)$.

\subsection{Linear analysis on ACyl manifolds}

We recall some results about linear analysis on ACyl Riemannian manifolds.
The references for the material in this subsection are \citet{Mazya1978} and \citet{Lockhart1985}.

\subsubsection{Translation-invariant operators on cylindrical manifolds}

Let $E\to X$ be a Riemannian vector bundle over a compact Riemannian manifold.
By slight abuse of notation we also denote by $E$ its pullback to $\R\times X$.
Denote by $t$ the coordinate function on $\R$.
For $k\in\N_0$, $\alpha\in(0,1)$ and $\delta\in\R$ we define
\begin{equation*}
  \|\cdot\|_{C^{k,\alpha}_\delta} := \|e^{-\delta t}\cdot\|_{C^{k,\alpha}}
\end{equation*}
and denote by $C^{k,\alpha}_\delta(\R\times X,E)$ the closure of $C^\infty_0(\R\times X,E)$ with respect to this norm.
We set $C^\infty_\delta:=\bigcap_{k} C^{k,\alpha}_\delta$.
 
Let $D\co C^\infty(X,E)\to C^\infty(X,E)$ be a linear self-adjoint elliptic operator of first order.
The operator
\begin{equation*}
  L_\infty:=\del_t-D
\end{equation*}
extends to a bounded linear operator $L_{\infty,\delta}\co C^{k+1,\alpha}_\delta(\R\times X,E)\to C^{k,\alpha}_\delta(\R\times X,E)$.

\begin{theorem}[{\citet[Theorem 5.1]{Mazya1978}}]
  \label{thm:mazya-plamenevskii}
  The linear operator $L_{\infty,\delta}$ is invertible if and only if $\delta\notin \spec(D)$.
\end{theorem}

Elements $a\in\ker L_\infty$ can be expanded as
\begin{equation}
  \label{eq:expansion}
  a=\sum_{\delta\in\spec D} e^{\delta t} a_\delta
\end{equation}
where $a_\delta$ are $\delta$--eigensections of $D$, see~\cite[Section~3.1]{Donaldson2002}.
One consequence of this is the following result:

\begin{prop}
  \label{prop:translation-invariant-kernel}
  Denote by $\lambda_+$ and $\lambda_-$ the first positive and negative eigenvalue of $D$, respectively.
  If $a\in \ker L_\infty$ and
  \begin{equation*}
    a = O(e^{\delta t}) ~\text{as}~ t\to\infty
  \end{equation*}
  with $\delta<\lambda_+$, then there exists $a_0\in\ker D$
  such that
  \begin{equation*}
    \nabla^k(a-a_0) = O(e^{\lambda_-t}) ~\text{as}~ t\to\infty
  \end{equation*}
  for all $k\in\N_0$.
  If $a\in L^\infty(\R\times X,E)$, then $a=a_0$.
\end{prop}

\subsubsection[Asymptotically translation-invariant operators ...]{Asymptotically translation-invariant operators on ACyl manifolds}

Let $M$ be a Riemannian manifold together with a compact set $K\subset M$ and a diffeomorphism $\pi\co M\setminus K\to \R_+\times X$ such that the push-forward of the metric on $M$ is asymptotic to the metric on $\R_+\times X$, this means here and in what follows that their difference and all of its derivatives are $O(e^{\delta t})$ as $t\to\infty$ with $\delta<0$.
Let $F$ be a Riemannian vector bundle and let $\bar\pi\co F|_{M\setminus K} \to E$ be a bundle isomorphism covering $\pi$ such that the push-forward of the metric on $F$ is asymptotic to the metric on $E$.
Denote by $t\co M\to[1,\infty)$ a smooth positive function which agrees with $t\circ\pi$ on $\pi^{-1}([1,\infty)\times X)$.  We define
\begin{equation*}
  \|\cdot\|_{C^{k,\alpha}_\delta}:=\|e^{-\delta t}\cdot\|_{C^{k,\alpha}}
\end{equation*}
and denote by $C^{k,\alpha}_\delta(M,F)$ the closure of $C^\infty_0(M,F)$ with respect to this norm.  

Let $L\co C_0^\infty(M,E)\to C_0^\infty(M,E)$ be an elliptic operator asymptotic to $L_\infty=\del_t-D$, that is, the coefficients of the push-forward of $L$ to $\R_+\times X$ are asymptotic to the coefficients of $L_\infty$.
The operator $L$ extends to a bounded linear operator $L_\delta\co C^{k+1,\alpha}_{\delta}(M,E)\to C^{k,\alpha}_{\delta}(M,E)$.

\begin{prop}[{\cite[Proposition 2.4]{Haskins2012}}]
  \label{prop:non-critical-fredholm}
  If $\delta\notin\spec(D)$, then $L_\delta$ is Fredholm.
\end{prop}

Elements in the kernel of $L$ still have an asymptotic expansion analogous to \eqref{eq:expansion}.
We need the following result which extracts the constant term of this expansion.

\begin{prop}
  \label{prop:abstract-iota}
  There is a constant $\delta_0>0$ such that, for all $\delta\in[0,\delta_0]$, $\ker L_\delta=\ker L_0$ and there is a linear map $\iota\co \ker L_0 \to \ker D$ such that
  \begin{equation*}
    \nabla^k\(\bar\pi_* a-\iota(a)\) = O(e^{-\delta_0 t})~\text{as}~t\to\infty
  \end{equation*}
  for all $k\in\N_0$; in particular,
  \begin{equation*}
    \ker\iota = \ker L_{-\delta_0}.
  \end{equation*}
\end{prop}

\begin{proof}
  Let $\lambda_\pm$ be the first positive/negative eigenvalue of $D$. 
  Pick $0<\delta_0<\min\(\lambda_+,-\lambda_-\)$ such that the decay conditions made above hold with $-2\delta_0$ instead of $\delta$.
  Given $a\in \ker L_{\delta_0}$, set $\tilde a := \chi(t) \bar\pi_* a_\pm$ with $\chi$ as in \autoref{f:chi}.
  Then $L_\infty\tilde a \in C^\infty_{-\delta_0}$.
  By \autoref{thm:mazya-plamenevskii} there exists a unique $b\in C^\infty_{-\delta_0}$ such that $L_\infty (\tilde a-b)=0$.
  By \autoref{prop:translation-invariant-kernel} $(\tilde
  a-b)_0\in \ker D$ and $\tilde a - b - (\tilde a-b)_0 =
  O(e^{\lambda_-t})$ as $t$ tends to infinity.
  From this it follows that $a\in\ker L_0$; hence, the first part of the proposition.
  With $\iota(a):=(\tilde a-b)_0$ the second part also follows.
\end{proof}

\subsection{Hermitian--Yang--Mills connections over Calabi--Yau \texorpdfstring{$3$}{3}--folds}

Suppose $(Z,\omega,\Omega)$ is Calabi--Yau $3$--fold and $(Y:=\R\times Z,\phi:=\rd t\wedge\omega+\Re\Omega)$ is the corresponding cylindrical $\Gtwo$--manifold.
In this section we relate translation-invariant $\Gtwo$--instantons over $Y$ with Hermitian--Yang--Mills connections over $Z$.
Let $G$ denote a compact semi-simple Lie group.

\begin{definition}
  Let $(Z,\omega)$ be a Kähler manifold and let $E$ be a $G$--bundle over $Z$.
  A connection $A$ on $E$ is \defined{Hermitian--Yang--Mills (HYM) connection} if
  \begin{equation}
    \label{eq:hym}
    F_A^{0,2} = 0 \qandq \Lambda F_A = 0.
  \end{equation}
  Here $\Lambda$ is the dual of the Lefschetz operator $L:=\omega\wedge\cdot$.
\end{definition}

\begin{remark}
  We are mostly interested in the special case of $\U(n)$--bundles;
  however, for $G=\U(n)$, \eqref{eq:hym} is too restrictive as it forces $c_1(E)=0$.
  There are two customary ways to circumnavigate this issue:
  One is to change \eqref{eq:hym} and instead of the second part require that $\Lambda F_A$ be equal to a constant in $\fu(1)$, the centre of $\fu(n)$, which is determined by the degree of $\det E$;
  the other one is to work with the induced $\PU(n)$--bundle.
  These view points are essentially equivalent and we adopt the latter.
\end{remark}

\begin{remark}
  By the first part of \eqref{eq:hym} a HYM connection induces a holomorphic structure on $E$.
  If $Z$ is compact, then there is a one-to-one correspondence between gauge equivalence classes of HYM connections on $E$ and isomorphism classes of polystable holomorphic $G^\C$--bundles $\sE$ whose underlying topological bundle is $E$, see Donaldson~\cite{Donaldson1985} and Uhlenbeck--Yau~\cite{Uhlenbeck1986}.
\end{remark}

On a Calabi--Yau $3$--fold \eqref{eq:hym} is equivalent to
\begin{equation*}
  F_A\wedge\Im\Omega=0 \qandq F_A\wedge\omega\wedge\omega=0;
\end{equation*}
hence, using $\psi=*\phi=*(\rd t\wedge \omega+\Re\Omega)=\frac12\omega\wedge\omega-\rd t\wedge\Im\Omega$ one easily derives:

\begin{prop}[{\cite[Proposition 8]{SaEarp2011}}]
  \label{prop:HYM-G2-instanton}
  Denote by $\pi_Z\co Y\to Z$ the canonical projection.
  $A$ is a HYM connection if and only if $\pi_Z^*A$ is a $\Gtwo$--instanton.
\end{prop}

In general, if $A$ is a $\Gtwo$--instanton on a $G$--bundle $E$ over a $\Gtwo$--manifold $(Y,\phi)$, then the moduli space $\sM$ of $\Gtwo$--instantons near $[A]$, i.e., the space of gauge equivalence classes of $\Gtwo$--instantons near $[A]$ is the space of small solutions $(\xi,a)\in\(\Omega^0\oplus\Omega^1\)(Y,\fg_E)$ of the system of equations
\begin{equation*}
  \rd_{A}^*a = 0 \qandq
  \rd_{A+a}\xi + *(F_{A+a}\wedge\psi) = 0
\end{equation*}
modulo the action of $\Gamma_A\subset \sG$, the stabiliser of $A$ in the gauge group of $E$---assuming $Y$ is compact or appropriate assumptions are made regarding the growth of $\xi$ and $a$.
The linearisation  $L_A\co\left(\Omega^0\oplus\Omega^1\right)(Y,\fg_E)\to\left(\Omega^0\oplus\Omega^1\right)(Y,\fg_E)$ of this equation is
\begin{equation}
  \label{eq:L}
  L_A := \begin{pmatrix}
    & \rd_A^* \\
    \rd_A & *(\psi\wedge\rd_A)
  \end{pmatrix}.
\end{equation}
It controls the infinitesimal deformation theory of $A$.

\begin{definition}
  \label{def:irreducible-unobstructed}
  $A$ is called \defined{irreducible and unobstructed} if $L_A$ is surjective.
\end{definition}

If $A$ is irreducible and unobstructed, then $\sM$ is smooth at $[A]$.
If $Y$ is compact, then $L_A$ has index zero; hence, is surjective if and only if it is invertible; therefore, irreducible and unobstructed $\Gtwo$--instantons form isolated points in $\sM$.
If $Y$ is non-compact, the precise meaning of $\sM$ and $L_A$ depends on the growth assumptions made on $\xi$ and $a$ and $\sM$ may very well be positive-dimensional.

\begin{prop}
  \label{prop:L-dt-D}
  If $A$ is HYM connection on a bundle $E$ over a $\Gtwo$--manifold $Y:=\R\times Z$ as  in \autoref{prop:HYM-G2-instanton}, then the operator $L_{\pi_Z^*A}$ defined in \eqref{eq:L} can be written as
  \begin{equation*}
    L_{\pi_Z^*A}= \tilde I \del_t
    + D_{A}
  \end{equation*}
  where 
  \begin{equation*}
    \tilde I:=
    \begin{pmatrix}
       & -1 & \\
      1 &   & \\
       && I
    \end{pmatrix}
  \end{equation*}
  and $D_A\co\left(\Omega^0\oplus\Omega^0\oplus\Omega^1\right)(Z,\fg_E)
  \to \left(\Omega^0\oplus\Omega^0\oplus\Omega^1\right)(Z,\fg_E)$
  is defined by
  \begin{equation}
    \label{eq:D}
    D_A:=
    \begin{pmatrix}
      && \rd_A^*\\
      && \Lambda\rd_A\\
      \rd_A & -I\rd_A & -*(\Im\Omega\wedge\rd_A)
    \end{pmatrix}.
  \end{equation}
  (Note that $TY = \underline\R\oplus \pi_Z^*TZ$.)
\end{prop}

\begin{proof}
  Plugging $\psi=\frac12\omega\wedge\omega - \rd t\wedge\Im\Omega$ into the definition of $L_{\pi_Z^*A}$ and using the fact that the complex structure acts via
  \begin{equation}\label{eq:I-1-forms}
    I=\frac12*(\omega\wedge\omega\wedge\cdot)
  \end{equation}
  on $\Omega^1(Z,\fg_E)$ the assertion follows by a direct computation.
\end{proof}

\begin{definition}
  \label{def:cH}
  Let $A$ be a HYM connection on a $G$--bundle $E$ over a Kähler
  manifold $(Z,\omega)$.  Set
  \begin{equation*}
    \cH^i_A := \ker\(\delbar_A\oplus\delbar_A^*\co\Omega^{0,i}\(Z,\fg^\C_E\)\to\(\Omega^{0,i+1}\oplus\Omega^{0,i-1}\)\(Z,\fg^\C_E\)\).
  \end{equation*}
  $\cH^0_A$ is called the space of \defined{infinitesimal automorphisms} of $A$ and $\cH^1_A$ is called the space of \defined{infinitesimal deformations} of $A$.
\end{definition}

\begin{remark}
  If $Z$ is compact and $A$ is a connection on a $\PU(n)$--bundle $E$ corresponding to a holomorphic vector bundle $\sE$, then $\cH^i_A\iso H^i(Z,\sEnd_0(\sE))$.
\end{remark}

\begin{prop}
  \label{prop:kerDA-Hi}
  If $(Z,\omega,\Omega)$ is a compact Calabi--Yau $3$--fold and $A$ is a HYM connection on a $G$--bundle $E\to Z$, then
  \begin{equation*}
    \ker D_A\iso\cH^0_A\oplus\cH^1_A
  \end{equation*}
  where $D_A$ is as in \eqref{eq:D}.
\end{prop}

\begin{proof}
  If $s\in\cH^0_A$ and $\alpha\in\cH^1_A$, then $D_A(\Re s,\Im s,\alpha+\bar\alpha)=0$.
  Conversely, if $(\xi,\eta,a)\in\ker D_A$, then applying $\rd_A^*$ (resp.~$\rd_A^*\circ I$) to
  \begin{equation*}
    \rd_A\xi-I\rd_A\eta-*(\Im\Omega\wedge\rd_Aa)=0,
  \end{equation*}
  using \eqref{eq:I-1-forms}, taking the $L^2$ inner product with $\xi$ (resp.~$\eta$) and integrating by parts yields $\rd_A\xi=0$ (resp.~$\rd_A\eta=0$).
  Thus $\xi+i\eta\in\cH^0_A$ and
  \begin{equation*}
    \rd_A^* a=0, \quad \Lambda\rd_A a=0 \quad\text{and}\quad \Im\Omega\wedge\rd_A a=0
  \end{equation*}
  which implies $\alpha:=a^{0,1}\in\cH^1_A$ because $\rd_A^* = \del_A^* + \delbar_A^*$ and $\Lambda\rd_A = -i \del_A^* + i\delbar_A^*$.
\end{proof}

\subsection{\texorpdfstring{$\Gtwo$}{G2}--instantons over ACyl \texorpdfstring{$\Gtwo$}{G2}--manifolds}

\begin{definition}
  \label{def:G2instanton_asymptotic}
  Let $(Y,\phi)$ be an ACyl $\Gtwo$--manifold with asymptotic cross-section $(Z,\omega,\Omega)$.
  Let $A_\infty$ be a HYM connection on a $G$--bundle $E_\infty\to Z$.  A $\Gtwo$--instanton $A$ on a $G$--bundle $E\to Y$ is called \defined{asymptotic} to $A_\infty$ if there exist a constant $\delta<0$ and a bundle isomorphism $\bar\pi\co E|_{Y\setminus K} \to E_\infty$ covering $\pi\co Y\setminus K \to \R_+\times Z$ such that
  \begin{equation}
    \label{eq:connection-decay}
    \nabla^k (\bar\pi_*A-A_\infty) = O(e^{\delta t})
  \end{equation}
  for all $k\in\N_0$.
  Here by a slight abuse of notation we also denote by $E_\infty$ and $A_\infty$ their respective pullbacks to $\R_+\times Z$.
\end{definition}

\begin{definition}
  \label{def:ag2i}
  Let $(Y,\phi)$ be an ACyl $\Gtwo$--manifold and let $A$ be a  $\Gtwo$--instanton on a $G$--bundle over $(Y,\phi)$ asymptotic to $A_\infty$.
  For $\delta\in\R$ we set
  \begin{equation*}
    \cT_{A,\delta} := \ker L_{A,\delta} = \left\{ \ua \in \ker L_A : \nabla^k \bar\pi_* \ua = O(e^{\delta t}) ~\text{for all}~k\in\N_0 \right\}.
  \end{equation*}
  where $\ua=(\xi,a)\in\left(\Omega^0\oplus\Omega^1\right)(Y,\fg_E)$.
  Set $\cT_A:=\cT_{A,0}$.
\end{definition}

\begin{prop}
  \label{prop:iota}
  Let $(Y,\phi)$ be an ACyl $\Gtwo$--manifold and let $A$ be a $\Gtwo$--instanton asymptotic to $A_\infty$.
  Then there is a constant $\delta_0>0$ such that for all $\delta\in[0,\delta_0]$, $\cT_{A,\delta}=\cT_A$ and there is a linear map $\iota\co \cT_{A} \to \cH^0_{A_\infty}\oplus \cH^1_{A_\infty}$ such that
  \begin{equation*}
    \nabla^k \(\bar\pi_* \ua-\iota(\ua)\)=O(e^{-\delta_0 t})
  \end{equation*}
  for all $k\in\N_0$; in particular,
  \begin{equation*}
    \ker\iota=\cT_{A,-\delta_0}.
  \end{equation*}
\end{prop}

\begin{proof}
  By \autoref{prop:L-dt-D}, $L_A$ is asymptotic to $\tilde I(\del_t-\tilde I D_A)$.
  Since $\tilde I D_A$ is self-adjoint and $\ker \tilde I D_A = \ker D_A$, we can apply \autoref{prop:abstract-iota} to obtain a linear map $\iota\co \cT_{A} \to \ker D_{A_\infty}$ and use the isomorphism $\ker D_{A_\infty} \iso \cH^0_{A_\infty}\oplus \cH^1_{A_\infty}$ from \autoref{prop:kerDA-Hi}.
\end{proof}

\begin{prop}
  \label{prop:lagrangian}
  Let $(Y,\phi)$ be an ACyl $\Gtwo$--manifold and let $A$ be a $\Gtwo$--instanton asymptotic to $A_\infty$.  
  Then
  \begin{equation*}
    \dim\im\iota =  \frac12 \dim\(\cH^0_{A_\infty}\oplus\cH^1_{A_\infty}\)
  \end{equation*}
  and, if $\cH^0_{A_\infty}=0$, then $\im\iota\subset\cH^1_{A_\infty}$ is Lagrangian with respect to the symplectic structure on $\cH^1_{A_\infty}$ induced by $\omega$.
\end{prop}

\begin{proof}
  By \citet[Theorem~7.4]{Lockhart1985} for $0<\delta\ll 1$
 \begin{equation*}
   \dim \im\iota = \ind L_{A,\delta} = \frac12\dim\ker D_{A_\infty}.
 \end{equation*}
  Suppose $\cH^0_{A_\infty}=0$. 
  If $(\xi,a) \in \cT_A$, then $\rd_A^*\rd_A\xi=0$ and, by \autoref{prop:iota}, $\xi$ decays exponentially.
  Integration by parts shows that $\rd_A\xi=0$;
  hence, $\xi=0$.
  Therefore, $\cT_A\subset\Omega^1(Y,\fg_E)$

  We show that $\im\iota$ is isotropic:
  For $a,b\in\cT_A$
  \begin{equation*}
    \frac12 \int_Z \<\iota(a)\wedge\iota(b)\>\wedge\omega\wedge\omega = \int_Y \rd \(\<a\wedge b\>\wedge\psi\) = 0
  \end{equation*}
  because $\rd_A a\wedge \psi=\rd_A b\wedge\psi = 0$.
\end{proof}

\subsection{Gluing \texorpdfstring{$\Gtwo$}{G2}--instantons over ACyl \texorpdfstring{$\Gtwo$}{G2}--manifolds}

In the situation of \autoref{prop:lagrangian}, if $\ker\iota=0$ and $\cH^0_{A_\infty}=0$, then one can show that the moduli space $\sM(Y)$ of $\Gtwo$--instantons near $[A]$ which are asymptotic to some HYM connection is smooth.
Although the moduli space $\sM(Z)$ of HYM connections near $[A_\infty]$ is not necessarily smooth, formally, it still makes sense to talk about its symplectic structure and view $\sM(Y)$ as a Lagrangian submanifold.
The following theorem shows, in particular, that transverse intersections of a pair of such Lagrangians give rise to $\Gtwo$--instantons.

\begin{theorem}
  \label{thm:gluing}
  Let $(Y_\pm,\phi_\pm)$ be a pair of ACyl $\Gtwo$--manifolds that match via $f\co Z_+\to Z_-$.
  Denote by $(Y_T,\phi_T)_{T\geq T_0}$ the resulting family of compact $\Gtwo$--manifolds arising from the construction in \autoref{sec:gluing-acyl-g2-manifolds}.
  Let $A_\pm$ be a pair of $\Gtwo$--instantons on $E_\pm$ over $(Y_\pm,\phi_\pm)$ asymptotic to $A_{\infty,\pm}$.
  Suppose that the following hold:
  \begin{itemize}
  \item There is a bundle isomorphism $\bar f\co E_{\infty,+}\to E_{\infty,-}$  covering $f$ such that $\bar f^*A_{\infty,-}=A_{\infty,+}$,
  \item The maps $\iota_\pm\co \cT_{A_\pm}\to \ker D_{A_{\infty,\pm}}$ constructed in \autoref{prop:iota} are injective and their images intersect trivially
    \begin{equation}
      \label{eq:trivial-intersection-2}
     \im\(\iota_+\) \cap \im\(\bar f^*\circ\iota_-\)=\{0\} \subset \cH^0_{A_{\infty,+}}\oplus \cH^1_{A_{\infty,+}}.
    \end{equation}
  \end{itemize}
  Then there exists $T_1\geq T_0$ and for each $T\geq T_1$ there exists an irreducible and unobstructed $\Gtwo$--instanton $A_T$ on a $G$--bundle\ $E_T$ over $(Y_T,\phi_T)$.
\end{theorem}

\begin{proof}
  The proof proceeds in three steps.
  We first produce an approximate $\Gtwo$--instanton $\tilde A_T$ by an explicit cut-and-paste procedure.
  This reduces the problem to solving the non-linear partial differential equation
  \begin{equation}
    \label{eq:perturbation}
    \rd_{\tilde A_t}^* a = 0 \qandq
    \rd_{\tilde A_T+a} \xi + *_T(F_{\tilde A_T + a}\wedge \psi_T) = 0.
  \end{equation}
  for $a\in \Omega^1(Y_T,\fg_{E_T})$ and $\xi\in\Omega^0(Y_T,\fg_{E_T})$ where $\psi_T:=*\phi_T$.
  Under the hypotheses of \autoref{thm:gluing} we will show that we can solve the linearisation of \eqref{eq:perturbation} in a uniform fashion.
  The existence of a solution of \eqref{eq:perturbation} then follows from a simple application of Banach's fixed-point theorem.

  \setcounter{step}{0}
  \begin{step}
    \label{step:1}
    There exists a $\delta < 0$ and for each $T\geq T_0$ there exists a connection $\tilde A_T$ on a $G$--bundle $E_T$ over $Y_T$ such that
    \begin{equation}
      \label{eq:error-estimate}
      \|F_{\tilde A_T}\wedge\psi_T\|_{C^{0,\alpha}}=O(e^{\delta T}).
    \end{equation}
  \end{step}

  The bundle $E_T$ is constructed by gluing $E_\pm|_{Y_{T,\pm}}$ via $\bar f$ and the connection $\tilde A_T$ is defined by
  \begin{equation*}
    \tilde A_T := A_{\pm} - \bar\pi_\pm^*[\chi(t-T+1) a_\pm]
  \end{equation*}
  over $Y_{T,\pm}$ where
  \begin{equation*}
    a_\pm := \bar\pi_{\pm,*}A_\pm - A_{\infty,\pm},
  \end{equation*}
  $\bar\pi_\pm$ is as in \autoref{def:G2instanton_asymptotic} and $\chi$ is as in \autoref{sec:gluing-acyl-g2-manifolds}.
  Then \eqref{eq:error-estimate} is a straight-forward consequence of \eqref{eq:eta-estimate} and \eqref{eq:connection-decay}.

  \begin{step}
    \label{step:2}
    Define a linear operator $L_T\co C^{1,\alpha}\to C^{0,\alpha}$ by \eqref{eq:L} with $A=\tilde A_T$ and $\phi=\phi_T$.
    Then there exist constants $\tilde T_1,c>0$ such that for all  $T\geq\tilde T_1$ the operator $L_T$ is invertible and
    \begin{equation}
      \label{eq:inverse-estimate}
      \|L_T^{-1}\ua\|_{C^{1,\alpha}}\leq ce^{\frac{|\delta|}{4}T} \|\ua\|_{C^{0,\alpha}}.
    \end{equation}
  \end{step}
   
  \setcounter{substep}{0}
  \begin{substep}
    \label{step:2-1}
    There exists a constant $c>0$ such that for all $T\geq T_0$
    \begin{equation}
      \label{eq:schauder-estimate}
      \|\ua\|_{C^{1,\alpha}}\leq c\(\|L_T\ua\|_{C^{0,\alpha}}+\|\ua\|_{L^\infty}\).
    \end{equation}
  \end{substep}

  This is an immediate consequence of standard interior Schauder estimates because of \eqref{eq:eta-estimate} and \eqref{eq:connection-decay}.

  \begin{substep}
    \label{step:2-2}
    There exist constants $\tilde T_1\geq T_0$ and $c>0$ such that for $T\in[\tilde T_1,\infty)$
    \begin{equation}
      \label{eq:L-infty-estimate}
      \|\ua\|_{L^\infty}\leq ce^{\frac{|\delta|}{4}T} \|L_T\ua\|_{C^{0,\alpha}}.
    \end{equation}
  \end{substep}

  Suppose not; then there exist a sequence $(T_i)$ tending to infinity and a sequence $(\ua_i)$ such that
  \begin{equation}
    \label{eq:Lua-conditions}
    \|\ua_i\|_{L^\infty}=1 \quad\text{and}\quad
    \lim_{i\to\infty} e^{\frac{|\delta|}{4}T_i}\|L_{T_i}\ua_i\|_{C^{0,\alpha}}=0.
  \end{equation}
  Then by \eqref{eq:schauder-estimate}
  \begin{equation}
    \label{eq:ua-conditions}
    \|\ua_i\|_{C^{1,\alpha}} \leq 2c.
  \end{equation}
  Hence, by Arzelà--Ascoli we can assume (passing to a subsequence) that the sequence $\ua_i|_{Y_{T_i,\pm}}$ converges in $C^{1,\alpha/2}_\loc$ to some section $\ua_{\infty,\pm}$ of $\(\Lambda^0\oplus\Lambda^1\)\otimes \fg_{E_\pm}$ over $Y_\pm$, which is bounded and satisfies
  \begin{equation*}
    L_{A_\pm} \ua_{\infty,\pm}=0
  \end{equation*}
  because of \eqref{eq:eta-estimate} and \eqref{eq:connection-decay}.
  Using standard elliptic estimates it follows that $\ua_{\infty,\pm}\in\cT_{A_\pm}$.

  \begin{prop}
    \label{prop:uniform-convergence}
    In the above situation
    \begin{equation*}
      \lim_{i\to\infty}\left\|(\ua_i|_{Y_{T_i,\pm}})-(\ua_{\infty,\pm}|_{Y_{T_i,\pm}})\right\|_{L^\infty(Y_{T_i,\pm})}=0.
    \end{equation*}
  \end{prop}

  The proof of this proposition will be given at the end of this section.
  Accepting it as a fact for now, it follows immediately that
  \begin{equation*}
    \iota_+(\ua_{\infty,+})=\bar f^*\circ\iota_-(\ua_{\infty,-})
  \end{equation*}
  because $Y_{T_i,+}\cap Y_{T_i,-} = [T_i,T_i+1]\times Z_+$.
  Now, by \eqref{eq:trivial-intersection-2} we must have $\iota_\pm(\ua_{\infty,\pm})=0$; hence, $\ua_{\infty,\pm}=0$, since $\iota_\pm$ are injective.

  However, by \eqref{eq:Lua-conditions} there exist $x_i\in Y_{T_i}$ such that $|\ua_{T_i}|(x_i)=1$.  By passing to a further subsequence and possibly changing the rôles of $+$ and $-$ we can assume that each $x_i \in Y_{T_i,+}$; hence, by \autoref{prop:uniform-convergence}, $\ua_{\infty,+}\neq 0$, contradicting what was derived above.
  This proves \eqref{eq:L-infty-estimate}.

  \begin{substep}
    We complete the proof of \autoref{step:2}.
  \end{substep}

  Combining \eqref{eq:schauder-estimate} and \eqref{eq:L-infty-estimate} yields
  \begin{equation*}
    \|\ua\|_{C^{1,\alpha}}\leq ce^{\frac{|\delta|}{4}T}\|L_T\ua\|_{C^{0,\alpha}}.
  \end{equation*}
  Therefore, $L_T$ is injective; hence, also surjective since $L_T$ is formally self-adjoint. 

  \begin{step}
    There exists a constant $T_1\geq \tilde T_1$ and for each $T\geq T_1$ a smooth solution $\ua=\ua_T$ of \eqref{eq:perturbation} such that  $\lim_{T\to\infty}\|\ua_T\|_{C^{1,\alpha}}=0$.
  \end{step}

  We can write \eqref{eq:perturbation} as
  \begin{equation}
    \label{eq:perturbation-2}
    L_T \ua + Q_T(\ua) + \epsilon_T=0
  \end{equation}
  where $Q_T(\ua):=\frac12*_T([a\wedge a]\wedge\psi_T)+[a,\xi]$ and $\epsilon_T:=*_T(F_{\tilde A_T}\wedge\psi_T)$.
  We make the ansatz  $\ua=L_T^{-1}\ub$.
  Then \eqref{eq:perturbation-2} becomes
  \begin{equation}
    \label{eq:perturbation-3}
    \ub + \tilde Q_T(\ub) + \epsilon_T=0
  \end{equation}
  where $\tilde Q_T=Q_T\circ L_T^{-1}$.
  By \eqref{eq:inverse-estimate}
  \begin{equation*}
    \|\tilde Q_T(\ub_1)-\tilde Q_T(\ub_2)\|_{C^{0,\alpha}}\leq
    c e^{\frac{|\delta|}{2}T} (\|\ub_1\|_{C^{0,\alpha}}+\|\ub_2\|_{C^{0,\alpha}})
    \|\ub_1-\ub_2\|_{C^{0,\alpha}}
  \end{equation*}
  for some constant $c>0$ independent of $T\geq \tilde T_1$.
  By \autoref{step:1}, $\|\epsilon_T\|_{C^{0,\alpha}} = O(e^{\delta T})$.
  Now, \autoref{lem:CM2} yields the desired solution of \eqref{eq:perturbation-3} and thus of \eqref{eq:perturbation} provided $T\geq T_1$ for a suitably large $T_1\geq \tilde T_1$.
  By elliptic regularity $\ua$ is smooth.
\end{proof}

\begin{lemma}[{\citet[Lemma~7.2.23]{Donaldson1990}}]
  \label{lem:CM2}
  Let $X$ be a Banach space and let $T \co X\to X$ be a smooth map with $T(0)=0$.
  Suppose there is a constant $c>0$ such that
  \begin{equation*}
    \|Tx-Ty\|\leq c\(\|x\|+\|y\|\)\|x-y\|.
  \end{equation*}
  If $y\in X$ satisfies $\|y\|\leq\frac{1}{10c}$, then there exists a unique $x\in X$ with $\|x\|\leq\frac{1}{5c}$ solving
  \begin{equation*}
    x+Tx=y.
  \end{equation*}
  Moreover, this $x\in X$ satisfies $\|x\|\leq 2\|y\|$.
\end{lemma}

To complete the proof of \autoref{thm:gluing} it now remains to prove \autoref{prop:uniform-convergence} for which we require the following result.

\begin{prop}
  \label{prop:L-right-inverse}
  In the situation of \autoref{thm:gluing}, there is a $\gamma_0>0$ such that for each $\gamma\in(0,\gamma_0)$ the linear operator $L_{A_\pm}\co C^{1,\alpha}_\gamma\to C^{0,\alpha}_\gamma$ has a bounded right inverse.
\end{prop}

\begin{proof}
  By \autoref{prop:non-critical-fredholm}, $L_{A_\pm}\co C^{1,\alpha}_\gamma\to C^{0,\alpha}_\gamma$ is Fredholm whenever $\gamma>0$ is sufficiently small.
  The cokernel of $L_{A_\pm}$ can be identified to be $\cT_{A_\pm,-\gamma}$, which is trivial by hypothesis.
\end{proof}

\begin{proof}[Proof of \autoref{prop:uniform-convergence}]
  We restrict to the $+$ case; the $-$ case is identical.
  It follows from the construction of $\ua_{\infty,+}$ that for each fixed compact subset $K\subset Y_+$
  \begin{equation*}
    \lim_{i\to\infty} \left\|(\ua_i|_K)-(\ua_{\infty,+}|_K)\right\|_{L^\infty(K)}=0.
  \end{equation*}
  To strengthen this to an estimate on all of $Y_{T_i,+}$ the factor $e^{\frac{|\delta|}{4}T}$ in \eqref{eq:Lua-conditions} will be important, even though it is clearly not optimal.

  With $\chi$ as in \autoref{f:chi} define a cut-off function $\chi_T\co Y_+ \to [0,1]$ by $ \chi_T(x):=1-\chi(t_+(x)-\frac32 T)$.
  For each sufficiently small $\gamma>0$ we have
  \begin{equation*}
    \|L_{A_+} (\chi_{T_i} \ua_i)\|_{C^{0,\alpha}_\gamma(Y_+)}=O(e^{-\frac32\gamma T_i})
  \end{equation*}
  using the estimates \eqref{eq:eta-estimate}, \eqref{eq:connection-decay}, \eqref{eq:Lua-conditions} and \eqref{eq:ua-conditions}.
  Using \autoref{prop:L-right-inverse} we construct $\ub_i\in C^{1,\alpha}_\gamma$ such that $\ua_{\infty,+}^i := \chi_{T_i} \ua_i+\ub_i \in\cT_{A_+,\gamma}$ and $\|\ub_i\|_{C^{1,\alpha}_{0,\gamma}}=O(e^{-\frac32\gamma T_i})$.
  Hence,
  \begin{equation*}
    \left\|(\ua_i|_{Y_{T_i,+}})-(\ua_{\infty,+}^i|_{Y_{T_i,+}})\right\|_{L^\infty(Y_{T_i,+})} = O(e^{-\frac12\gamma T_i}).
  \end{equation*}
  Moreover, $\lim_{i\to\infty}\left\|(\ua_{\infty,+}^i|_K)-(\ua_{\infty,+}|_K)\right\|_{L^\infty(K)} = 0$ and since both $\|\cdot\|_{L^\infty(K)}$ and $\|\cdot\|_{L^\infty(Y_+)}$ are norms on the finite dimensional vector space $\cT_{A_+,\gamma}=\cT_{A_+}$ it also follows that
  \begin{equation*}
    \lim_{i\to\infty} \|\ua_{\infty,+}^i - \ua_{\infty,+}\|_{L^\infty(Y_+)} = 0.
  \end{equation*}
  Therefore,
  \begin{equation*}
    \lim_{i\to\infty}\left\|(\ua_i|_{Y_{T_i,+}})-(\ua_{\infty,+}|_{Y_{T_i,+}})\right\|_{L^\infty(Y_{T_i,+})} = 0.
    \qedhere
  \end{equation*}
\end{proof}

\begin{remark}
  \label{rmk:transversality}
  The proof of \autoref{thm:gluing} slightly simplifies assuming $\cH^0_{A_{\infty,+}}\oplus \cH^1_{A_{\infty,+}}=\{0\}$ instead of \eqref{eq:trivial-intersection-2}:
  We can directly conclude that $\iota_\pm(\ua_{\infty,\pm})=0$ and, hence, $\ua_{\infty,\pm}=0$; thus making \autoref{prop:uniform-convergence} unnecessary.
  In particular, \eqref{eq:L-infty-estimate} holds without the additional factor of $e^{\frac{|\delta|}{4}T}$.
\end{remark}


\section[From holomorphic bundles over building blocks to \texorpdfstring{$\Gtwo$}{G2}--instantons ...]{From holomorphic vector bundles over building blocks to \texorpdfstring{$\Gtwo$}{G2}--instantons over ACyl \texorpdfstring{$\Gtwo$}{G2}--manifolds}
\label{sec:holomorphic}

We now discuss how to deduce \autoref{thm:itcs} from \autoref{thm:gluing}.

\begin{definition}
  Let $(V,\omega,\Omega)$ be an ACyl Calabi--Yau $3$--fold with asymptotic cross-section $(\Sigma,\omega_I,\omega_J,\omega_K)$.
  Let $A_\infty$ be an ASD instanton on a $G$--bundle $E_\infty$ over $\Sigma$.
  A HYM connection $A$ on a $G$--bundle $E$ over $V$ is called \defined{asymptotic} to $A_\infty$ if there exist a constant $\delta<0$ and a bundle isomorphism $\bar\pi\co E|_{V\setminus K} \to E_\infty$ covering $\pi\co V\setminus K \to \R_+\times S^1\times \Sigma$ such that
  \begin{equation*}
    \nabla^k (\bar\pi_*A-A_\infty) = O(e^{\delta t})
  \end{equation*}
  for all $k\in\N_0$.
  Here by a slight abuse of notation we also denote by $E_\infty$ and $A_\infty$ their respective pullbacks to $\R_+\times S^1\times\Sigma$.
\end{definition}

The following theorem can be used to produce examples of HYM connections $A$ on $\PU(n)$--bundles over ACyl Calabi--Yau $3$--folds asymptotic to ASD instantons $A_\infty$; hence, by taking the product with $S^1$, examples of $\Gtwo$--instantons $\pi_V^*A$ asymptotic to $\pi_\Sigma^*A_\infty$ over the ACyl $\Gtwo$--manifold $S^1\times V$.
Here $\pi_V\co S^1\times V \to V$ and $\pi_\Sigma \co T^2\times \Sigma\to \Sigma$ denote the canonical projections.

\begin{theorem}[{\citet[Theorem~59]{SaEarp2011}}]
  \label{thm:saearp}
  Let $Z$ and $\Sigma$ be as in \autoref{thm:hhn} and let $(V:=Z\setminus \Sigma,\omega,\Omega)$ be the resulting ACyl Calabi--Yau $3$--fold.
  Let $\sE$ be a holomorphic vector bundle over $Z$ and let $A_\infty$ be an ASD instanton on $\sE|_\Sigma$ compatible with the holomorphic structure.
  Then there exists a HYM connection $A$ on $\sE|_V$ which is compatible with the holomorphic structure on $\sE|_V$ and asymptotic to $A_\infty$.
\end{theorem}

By slight abuse of notation we also denote by $A_\infty$ the ASD instanton on the $\PU(n)$--bundle associated with $\sE|_\Sigma$ and by $A$ the HYM connection on the $\PU(n)$--bundle associated with $\sE|_V$.
\autoref{thm:gluing} and \autoref{thm:saearp} together with the following result immediately imply \autoref{thm:itcs}.

\begin{prop}
  \label{prop:diagram}
  In the situation of \autoref{thm:saearp}, suppose $H^0(\Sigma,\sEnd_0(\sE|_\Sigma))=0$.
  Then
  \begin{equation}
    \label{eq:H1-holomorphic-ASD}
    \cH^1_{\pi_\Sigma^*A_\infty} = H^1_{A_\infty},
  \end{equation}
  see \autoref{def:cH} and \autoref{rmk:H1-ASD}, and for some small $\delta>0$ there exist injective linear maps 
  \begin{align*}
    & \kappa_- \co \cT_{\pi_V^*A,-\delta}\to H^1(Z,\sEnd_0(\sE)(-\Sigma)) \\
    \andq& \kappa \co \cT_{\pi_V^*A}\to H^1(Z,\sEnd_0(\sE)) 
  \end{align*}   
  such that the following diagram commutes:
  \begin{equation}
  \label{eq:diagram}
    \begin{tikzpicture}[baseline=(current bounding box.center)]
      \matrix (m) [matrix of math nodes, row sep=3.2em, column
      sep=1.6em] {
        \cT_{\pi_V^*A,-\delta} & \cT_{\pi_V^*A} & \cH^1_{\pi_\Sigma^*A_\infty} \\
        H^1(Z,\sEnd_0(\sE)(-\Sigma)) &H^1(Z,\sEnd_0(\sE)) &
        H^1(\Sigma,\sEnd_0(\sE|_\Sigma)). \\}; \path[-stealth]
      (m-1-1) edge 
      (m-1-2) edge node [right] {$\kappa_{-}$} (m-2-1)
      (m-1-2) edge node [above] {$\iota$} (m-1-3)
              edge node [right] {$\kappa$} (m-2-2)
      (m-1-3) edge node [right] {$\iso$} (m-2-3)
      (m-2-1) edge (m-2-2)
      (m-2-2) edge (m-2-3);
    \end{tikzpicture}
  \end{equation}
\end{prop}

Equation \eqref{eq:H1-holomorphic-ASD} is a direct consequence of $\cH^0_{A_\infty}=0$.
The proof of the remaining assertions requires some preparation.

\subsection{Comparing infinitesimal deformations of \texorpdfstring{$\pi_V^*A$}{pi*A} and \texorpdfstring{$A$}{A}}

\begin{prop}
  \label{prop:T-kerD}
  If $A$ is a HYM connection asymptotic to $A_\infty$, then there exists a $\delta_0>0$ such that for all $\delta\leq \delta_0$
  \begin{equation}
    \label{eq:T-kerD}
    \cT_{\pi_V^*A,\delta} =\left\{ \ua \in \ker D_A : \nabla^k \bar\pi_*
      \ua=O(e^{\delta t}) ~\text{for all}~k\in\N_0 \right\}
  \end{equation}
  with $D_A$ as in \eqref{eq:D}.
\end{prop}

\begin{proof}
  We can write $L_A = \tilde I \del_\beta + D_A$ where $\beta$ denotes the coordinate on $S^1$.
  For $\delta\leq 0$, \eqref{eq:T-kerD} follows by an application of Lemma A.1 in \cite{Walpuski2011}.
  The right-hand side is contained in the left-hand side of \eqref{eq:T-kerD} which, by \autoref{prop:iota}, is independent of $\delta\in[0,\delta_0]$.
\end{proof}

\begin{prop}
  \label{prop:T-H}
  In the situation of \autoref{prop:diagram}, there exists a constant $\delta_0>0$ such that, for all $\delta\leq\delta_0$, $\cH^0_{A,\delta}=0$ and
  \begin{equation*}
    \cT_{\pi_V^*A,\delta}\iso \cH^1_{A,\delta}
  \end{equation*}
  where
  \begin{equation*}
    \cH^i_{A,\delta} :=\left\{ \alpha\in\cH^i_A : \nabla^k \bar\pi_*\alpha=O(e^{\delta t})~\text{for all}~k\in\N_0 \right\}.
  \end{equation*}
\end{prop}

\begin{proof}
  If $\delta\leq\delta_0$ (cf.~\autoref{prop:iota}) and $(\xi,\eta,a)\in\cT_{A,\delta}$, then $\iota(\xi,\eta,a)\in \{0\}\oplus\cH^1_{A_\infty}$.
  Hence $\xi$ and $\eta$ decay exponentially and one use can \autoref{prop:T-kerD} and argue as in the proof of \autoref{prop:kerDA-Hi}; it also follows that $\cH^0_{A,\delta}=0$.
\end{proof}

\subsection{Acyclic resolutions via forms of exponential growth/decay}

In view of the above what is missing to prove \autoref{prop:diagram} is a way to relate $\cH^1_{A,\delta}$ with the cohomology of (twists of) $\sEnd_0(\sE)$.
This is what the following result provides.

\begin{prop}
  \label{prop:resolution}
  Let $(Z,\Sigma)$ be a building block and let $V:=Z\setminus\Sigma$ be the ACyl Calabi--Yau $3$--fold constructed via \autoref{thm:hhn}.  Suppose that $\sE$ is a holomorphic vector bundle over $Z$ and suppose that $A$ is a HYM connection on $\sE$ compatible with the holomorphic structure and asymptotic to an ASD instanton on $\sE|_{\Sigma}$.

  For $\delta\in \R$ define a complex of sheaves $(\cA^\bullet_\delta,\delbar)$ on $Z$ by
  \begin{equation*}
    \cA^i_\delta(U)=\left\{\alpha\in \Omega^{0,i}\(V\cap U,\sE\) : \nabla^k \bar\pi_*\alpha = O(e^{\delta t}) ~\text{for all}~k\in\N_0 \right\}.
  \end{equation*}
  If $\delta\in\R\setminus\Z$, then the complex of sheaves $(\cA^\bullet_\delta,\delbar)$ is an acyclic resolution of $\sE\(\floor{\delta}\Sigma\)$.
  In particular, setting  $\kappa^i_\delta(\alpha):=[\alpha]$ one obtains maps
  \begin{equation*}
    \kappa^i_\delta\co\cH^i_{A,\delta} \to H^i\(\Gamma(\cA^\bullet_\delta),\delbar\)\iso H^i(Z,\sE(\floor{\delta}\Sigma)).
  \end{equation*}
\end{prop}

\begin{remark}
  In \autoref{prop:resolution}, $\floor{\delta}$ denotes the largest integer not greater than $\delta$;
  in particular, $\floor{\delta}\Sigma$ is a divisor on $Z$.
\end{remark}

\begin{remark}
  We state \autoref{prop:resolution} in dimension three; however, it works \emph{mutatis mutandis} in all dimensions.
\end{remark}

\begin{proof}[Proof of \autoref{prop:resolution}]
  The proof consists of three steps.

  \setcounter{step}{0}
  \begin{step}
    The sheaves $\cA^\bullet_\delta$ are $C^\infty$--modules; hence, acyclic, see \cite[Chapter IV Corollary 4.19]{Demailly2012}.
  \end{step}

  \begin{step}
    $\sE(\floor{\delta}\Sigma)=\ker\(\delbar\co \cA^0_\delta\to\cA^1_\delta\)$.
  \end{step}

  Let $x\in Z$ and let $U\subset Z$ denote a small open neighbourhood of $x$.
  An element $s\in\ker\(\delbar\co \Gamma(U,\cA^0_\delta)\to\Gamma(U,\cA^1_\delta)\)$ corresponds to a holomorphic section of $\sE|_{V\cap U}$ such that $|z|^{-\delta} s$ stays bounded.
  Here $z$ is a holomorphic function on $U$ vanishing to first order along $\Sigma\cap U$, whose existence follows from \autoref{def:building-block}.
  Then $z^{-\floor\delta} s$ is weakly holomorphic in $U$.
  By elliptic regularity $z^{-\floor\delta} s$ extends across $U\cap \Sigma$ and thus $s$ defines an element of $\Gamma\(U,\sE(\floor\delta \Sigma)\)$.
  Conversely, it is clear that $\Gamma(U,\sE(\floor\delta \Sigma))\subset\ker\(\delbar\co\Gamma(U,\cA^0_\delta)\to\Gamma(U,\cA^1_\delta)\)$.
    
  \begin{step}
   The complex of sheaves $\(\cA^\bullet_\delta,\delbar\)$ is exact.
  \end{step}

  Away from $\Sigma$ the exactness follows from the usual $\delbar$--Poincaré Lemma.
  If $x\in\Sigma$, then since $Z$ is fibred over $\P^1$, by \autoref{def:building-block}, there exist a small open neighbourhood $U$ of $x$ in $Z$, a polydisc $D\subset \Sigma$ centred at $x$ and a biholomorphic map $\pi\co V\cap U \to \R_+\times S^1\times D$ such that the push-forward of the Kähler metric on $V\cap U$ via $\pi$ is asymptotic to the metric induced by that on $D$.
  The necessary version of the $\delbar$--Poincaré Lemma can now be proved along the lines of \cite[p.~25]{Griffiths1994} provided the linear operator
  \begin{equation*}
    \delbar\co C^\infty_\delta\Omega^0(\R\times S^1) \to C^\infty_\delta\Omega^{0,1}(\R\times S^1)
  \end{equation*}
  is invertible.
  This, however, is a simple consequence of \autoref{thm:mazya-plamenevskii} since $\delbar = \del_t +i\del_\alpha$ and the spectrum of $i\del_\alpha$ on $S^1=\R/\Z$ is $\Z$.
\end{proof}

\subsection{Proof of \autoref{prop:diagram}}

In view of \autoref{prop:T-H} we only need to establish \eqref{eq:diagram} with $\cH^1_{A,\delta}$ instead of $\cT_{\pi_V^*A,\delta}$.  
By \autoref{prop:resolution} applied to $\sEnd_0(\sE)$, we have linear maps 
\begin{equation*}
  \kappa_\delta^1\co \cH^1_{A,\delta} \to H^1\(Z,\sEnd_0(\sE)(\floor\delta \Sigma)\) \quad\text{for}~\delta \in\R\setminus \Z;
\end{equation*}
hence, linear maps 
\begin{align*}
  & \kappa_-\co \cH^1_{A,-\delta}\to H^1\(Z,\sEnd_0(\sE)(-\Sigma)\) \\
  \andq& \kappa\co \cH^1_A=\cH^1_{A,\delta}\to H^1\(Z,\sEnd_0(\sE)\)
\end{align*}
for some small $\delta>0$ making the following diagram commute:
\begin{equation*}
  \begin{tikzpicture}[baseline=(current  bounding  box.center)]
    \matrix (m) [matrix of math nodes, row sep=3.2em, column
      sep=1.6em] {
      \cH^1_{A,-\delta} & \cH^1_{A} & H^1_{A_\infty} \\
      H^1(Z,\sEnd_0(\sE)(-\Sigma)) &H^1(Z,\sEnd_0(\sE)) &
      H^1(\Sigma,\sEnd_0(\sE|_\Sigma)). \\}; \path[-stealth]
    (m-1-1) edge 
    (m-1-2) edge node [right] {$\kappa_{-}$} (m-2-1)
    (m-1-2) edge node [above] {$\iota$} (m-1-3)
    edge node [right] {$\kappa$} (m-2-2)
    (m-1-3) edge node [right] {$\iso$} (m-2-3)
    (m-2-1) edge (m-2-2)
    (m-2-2) edge (m-2-3);
  \end{tikzpicture}
\end{equation*}
The map $\kappa_-$ is injective, because if $\kappa_- a=0$, then
$a=\delbar s$ for some $s\in\Gamma(Z,\cA^0_{-\delta})$ and thus
\begin{equation*}
  \int_V \|a\|^2 = \int_V \<a,\delbar s\> = \int_V \<\delbar^* a,s\> = 0.
\end{equation*}
Since $H^0(\Sigma,\sEnd_0(\sE|_\Sigma))=0$, the first map on the bottom is injective and because the rows are exact a simple diagram chase proves shows that $\kappa$ is injective. 
\qed


\printreferences

\end{document}
